\documentclass[12pt,twoside]{article}
\usepackage[ansinew]{inputenc}
\usepackage{amsfonts}
\usepackage{latexsym}
\usepackage{amsmath}
\usepackage{amssymb}
\usepackage{eepic}
\usepackage{graphicx}
\usepackage{float}
\usepackage{pstricks}
\usepackage{epic,eepic}

\usepackage{xcolor}
\usepackage{tikz}
\usetikzlibrary{arrows,shapes,chains}

\newcommand{\R}{\mathbb{R}}

\newcommand{\C}{\mathbb{C}}
\newcommand{\N}{\mathbb{N}}

\newcommand{\U}{\underline}

\newtheorem{lettertheorem}{Theorem}

\newtheorem{defin}{Definition}[section]

\newtheorem{theorem}[defin]{Theorem}

\newtheorem{exa}{Example}
\newenvironment{example}{\begin{exa}\rm}{\end{exa}}
\newtheorem{lemma}[defin]{Lemma}
\newtheorem{corollary}[defin]{Corollary}
\newenvironment{proof}
{\noindent{\it Proof.}}{\hfill $\Box$\par\vspace{2.5mm}}

\newtheorem{que}{Question}

\newtheorem{pro}{Problem}

\numberwithin{equation}{section}

\makeatletter
\renewcommand{\ps@myheadings}{%
\renewcommand{\@evenhead}%
{{\rm\thepage}\hfil{\sc Ishizaki and Wen}\hfil}%
\renewcommand{\@oddhead}%
{\hfil{{\sc Orders of solutions of linear difference equations}\hfil{\rm\thepage}}}%
\renewcommand{\@evenfoot}{}%
\renewcommand{\@oddfoot}{}%
}\makeatother 

\title{\bf\Large All possible orders less than 1 of transcendental entire solutions of linear difference equations with polynomial coefficients}
\author{Katsuya Ishizaki \footnote{\ Supported by JSPS KAKENHI Grant Number 20K03658} and Zhi-Tao Wen
\footnote{\ Supported by the National Natural Science Foundation of China (No.~11971288) and Shantou University SRFT (NTF18029)}
\footnote{Corresponding author.}}
\date{}
\begin{document}
\maketitle

\begin{abstract}

In this paper, we study all possible orders which are less than 1 of transcendental entire solutions of linear difference equations
	\begin{equation}
    P_m(z)\Delta^mf(z)+\cdots+P_1(z)\Delta f(z)+P_0(z)f(z)=0,\tag{+}
    \end{equation}
where $P_j(z)$ are polynomials for $j=0,\ldots,m$. Firstly, we give the condition on existence of transcendental entire solutions of order less than 1 of difference equations (+).
Secondly, we give a list of all possible orders which are less than 1 of transcendental entire solutions of difference equations (+). Moreover, the maximum number of distinct orders which are less than 1 of transcendental entire solutions of difference equations (+) are shown. In addition, for any given rational number $0<\rho<1$, we can construct
a linear difference equation with polynomial coefficients which has a transcendental entire solution of order $\rho$. At least, some examples are illustrated for our main theorems.


\medskip
\noindent
\textbf{2020MSC}: 39B32; 30D35.

\end{abstract}

\section{Introduction}

For a function $f$, we denote by $\Delta f(z)=f(z+1)-f(z)$ the difference operator. Let $n$ be a nonnegative integer. Define $\Delta^n f(z)=\Delta(\Delta^{n-1} f(z))$ for $n\geq 1$, and write $\Delta^0 f=f$.
It is well known that the linear difference equation of order $m$
    $$
    a_m(z)\Delta^mf(z)+\cdots+a_1(z)\Delta f(z)+a_0(z)f(z)=0
    $$
with entire coefficients $a_j$, $j=0,1,\ldots,m$, has a system of $m$ meromorphic solutions which are linearly independent over the field of periodic functions with period one, see \cite{Praagman}. In particular, if $a_j$, $j=0, 1, \dots, m$, are polynomials, i.e.,
    \begin{equation}\label{linear.eq}
    P_m(z)\Delta^mf(z)+\cdots+P_1(z)\Delta f(z)+P_0(z)f(z)=0,
    \end{equation}
the order of growth of entire solutions has been investigated in connection with Newton polygon,
where $P_j$, $j=0,1,\dots, m$, $P_m(z)\not\equiv0$, are polynomials with degree $d_j$ for $j=0, 1, \dots, m$.

The {\it Newton polygon} or {\it Newton-Puiseux diagram} has played important roles not only in the theory of linear differential equations but also in the theory of linear difference equations, see e.g., \cite{Birkhoff1930}, \cite[Section 163]{Norlund1924}.
In~\cite[Theorem~1.1]{IY2004},  Ishizaki and Yanagihara adopted the {\it Newton polygon} of \eqref{linear.eq} as the convex hull of ${\mathfrak N} = \bigcup_{j=0}^m {\mathfrak N}_j$, where ${\mathfrak N}_j = \{(x,y) \; ; \; x \geq j, \; y \leq d_{m-j} - (m-j) \}$ for $0 \leq j \leq m$, and established a difference version of Wiman-Valiron theory for entire functions to obtain Theorem~\ref{CF} below under the condition $\rho(f)<1/2$.
Chiang and Feng extended the condition on the order of growth to entire solutions of order strictly less than 1 by
establishing an another difference version of Wiman-Valiron theory in \cite{CF2009}.
 The more precise difference version of Wiman-Valiron estimate was given by Chiang and Feng in \cite{CF2016}, they conclude that
 entire solutions to linear difference equations with polynomial coefficients
of order strictly less than 1 must have completely regular growth of rational order.

\begin{lettertheorem}\label{CF}{\rm ~\cite[Theorem~~4]{CF2016}}\enspace
Let $f$ be an entire solution of order of growth $\rho(f)<1$.
Then the order of growth $\rho(f)$ is a rational number which can be determined from a gradient of the corresponding Newton polygon of \eqref{linear.eq}.
In particular,
\begin{equation*}
\log M(r,f) = Lr^{\rho(f)}(1+o(1)),
\end{equation*}
where $L>0$ and $M(r,f)=\max_{|z|=r}|f(z)|$. Moreover,
the solution has completely regular growth.
\end{lettertheorem}

The natural question arises from Theorem \ref{CF} whether all possible orders of entire solutions \eqref{linear.eq} are determined from
a gradient of the corresponding Newton polygon of \eqref{linear.eq}.
Unfortunately, the general situation is not so simple.
For example, there exists an entire solution $f(z)=1/\Gamma(z)$ of the difference equation $z\Delta f(z)+(z-1)f(z)=0$, which is of order $\rho(f)=1$.
In fact, the entire function $g(z)=\pi(z)f(z)$ is also the solution of $z\Delta f(z)+(z-1)f(z)=0$, where $\pi(z)$ is any entire
periodic function with period one. Note that
for any $\sigma\in[1,\infty)$, there exists a prime periodic entire function $\pi(z)$ of order $\rho(\pi)=\sigma$ by \cite[Theorem 1]{Ozawa}. Hence
the growth of order of $g$ is not determined by the equation.

However, the case of the entire solutions of order less than 1 is different. Suppose that $f_1,\ldots,f_n$ are entire functions of order less than 1 linearly independent over the periodic function field of period one, and $f=\pi_1f_1+\cdots+\pi_nf_n$, where $\pi_1,\ldots,\pi_n$
are periodic functions with period one. If $\rho(f)<1$,
then $\pi_1,\ldots,\pi_n$ are constants. Therefore,
all possible orders which are less than 1 of entire solutions to a linear difference equation with polynomial coefficients could be determined by the equation itself.

It inspires us to consider the entire solutions of order less than 1 of linear difference equations with polynomial coefficients. In fact, there exists
an entire solution of order less than 1 of linear difference equations with polynomial coefficients under some conditions. For example, there exists an entire solution $f$ of order $\rho(f)=1/2$ of the difference equation
    $$
    (4z+6)\Delta^2 f(z)+3\Delta f(z)+f(z)=0.
    $$

Wittich \cite[pp. 65--68]{Wittich} and Gundersen et al. \cite{GSW1998} discussed whether
the possible rational order given by the Newton polygon for the entire solutions of linear differential equations with polynomial coefficients could be attained or not. Chiang and Feng in \cite{CF2016} ask this question for difference equations \eqref{linear.eq}.
According to the discussion above, some questions can be asked. For example

\begin{itemize}
\item
Which condition implies \eqref{linear.eq} has at least one transcendental entire solution of order less than 1?
\item
Is it possible to give a list of all possible orders which are less than 1 of transcendental entire solutions of a given difference equation \eqref{linear.eq}?
\item
What is the maximum number of distinct orders which are less than 1 of transcendental entire solutions of a given difference equation \eqref{linear.eq}?
%\item
%What is the maximum number of polynomial solutions of
%a given difference equation \eqref{linear.eq}?
\end{itemize}

In this paper, we answer these three questions by using binomial series, which is a different statement and method from the paper of Gundersen et al. \cite{GSW1998}. We also response Chiang and Feng's concerns in \cite{CF2016}.
Moreover, for any given rational number $0<\rho<1$, we show that it is possible to construct a linear difference equation with polynomial coefficients which has at least one transcendental entire solution of order $\rho$. At least, several examples are given to illustrate our results.

\section{Statement of results}

Consider equation \eqref{linear.eq}. For convenience,
we set $d_j=\deg P_j$. We define a strictly decreasing finite sequence of non-negative integers
    \begin{equation}\label{sequence.eq}
    s_1>s_2>\cdots>s_p\geq 0
    \end{equation}
in the following manner. We choose $s_1$ to be the unique integer satisfying
    \begin{equation}\label{s1.eq}
    d_{s_1}=\max_{0\leq k\leq m} {d_k} \quad\text{and}\quad d_{s_1}>d_{k}\quad\text{for all}~0\leq k<s_1.
    \end{equation}
Then given $s_j$, $j\geq 1$, we define $s_{j+1}$ to be the unique integer satisfying
    \begin{equation}\label{sj.eq}
    d_{s_{j+1}}-s_{j+1}> d_{s_{j}}-s_{j}
    \end{equation}
and
    \begin{equation}\label{sj2.eq}
    d_{s_{j+1}}=\max_{0\leq k<s_j} {d_k} \quad\text{and}\quad d_{s_{j+1}}>d_{k}\quad\text{for all}~0\leq k<s_{j+1}.
    \end{equation}
For a certain $p$, the integer $s_p$ will exist, but the integer $s_{p+1}$ will not exist, and the sequence $s_1,s_2,\ldots,s_p$ terminates with $s_p$.
Obviously, $1\leq p\leq n$ and \eqref{sequence.eq} holds.

We mention that the integers $s_1,\ldots,s_p$ in \eqref{sequence.eq} could also be expressed in the following manner:
    $$
    s_1=\min\left\{j:~d_j=\max_{0\leq k\leq m}{d_k}\right\}
    $$
and given $s_{j}$ and $j\geq 1$, we define
    $$
    s_{j+1}=\min\left\{i:~ d_{s_{i}}-s_{i}> d_{s_{j}}-s_{j}\quad\text{and}\quad d_i=\max_{0\leq k<s_j}{d_k} \right\}.
    $$

From the definition of the sequence \eqref{sequence.eq},
we see that
    $
    d_{s_1}>d_{s_2}>\cdots>d_{s_p}.
    $
Moreover, it is obvious that
    $$
    d_{s_p}-s_p>\cdots>d_{s_2}-s_2>d_{s_1}-s_1.
    $$

Correspondingly, we define $j=1,2,\ldots,p-1$
    \begin{equation}\label{orderlist.eq}
    \rho_j=1+\frac{d_{s_{j+1}}-d_{s_{j}}}{s_{j}-s_{j+1}}
    \end{equation}
when $p\geq 2$. From \eqref{sequence.eq} to \eqref{orderlist.eq}, we observe that $0<\rho_j<1$ for each $j$, $1\leq j\leq p-1$. Moreover, we see that
    $$
    1>\rho_1>\rho_2>\cdots>\rho_p>0.
    $$

\bigskip

In order to answer our four questions in the Section 1. We state our results as follows

\begin{theorem}\label{one.theorem}
Suppose that $p=1$. There does not exist any transcendental entire solution
of order less than 1 of difference equations \eqref{linear.eq} with polynomial coefficients.
\end{theorem}

\begin{theorem}\label{two.theorem}
If $p\geq 2$, then there exists at least one transcendental entire solution of order less than 1 of difference equations \eqref{linear.eq}. Moreover,
 there exist at lease one and at most $(d_{s_{j+1}}-s_{j+1})-(d_{s_j}-s_j)$ linearly independent solutions of order less than 1 of \eqref{linear.eq} with polynomial coefficients
such that $\rho(f)=\rho_j$ for $j=1,2,\ldots,p-1$.
\end{theorem}

We obtain the following corollaries from Theorem \ref{one.theorem} and Theorem \ref{two.theorem}.
\begin{corollary}
There exists at least one transcendental entire solution of order less than 1 of difference equations \eqref{linear.eq} if and only if $p\geq 2$.
\end{corollary}

\begin{corollary}
If $p\geq 2$, then there exist at most $(s_1-s_p)-(d_{s_1}-d_{s_p})$ transcendental entire solutions of order less than 1 of difference equations \eqref{linear.eq}.
\end{corollary}

\begin{corollary}
There does not exist any transcendental entire solution $f$ of order $\rho(f)=0$ of difference equations \eqref{linear.eq}.
\end{corollary}

\section{Binomial series}

We recall and study the properties of binomial series,~\cite{IW2021},~\cite{IY2004}.
Define $z^{\underline{0}}=1$ and
\begin{equation}
z^{\underline{n}}=z(z-1)\cdots(z-n+1)=n!\begin{pmatrix}
z\\
n
\end{pmatrix},\quad n=1, 2, 3, \dots,\label{1.3}
\end{equation}
which is called a {\it falling factorial}.
This yields
\begin{equation}
\Delta z^{\underline{n}}=(z+1)^{\underline{n}}-z^{\underline{n}}=nz^{\underline{n-1}}\, ,\label{21.02}
\end{equation}
which corresponds to $(z^n)'=nz^{n-1}$ in the differential calculus. Consider the formal series of the form
\begin{equation}
Y(z)=\sum_{n=0}^\infty a_n z^{\underline{n}},\quad a_n\in\mathbb C,\quad n=0,1,2, \dots.\label{1.5}
\end{equation}
For a fixed $z$, if $\sum_{n=0}^\infty |a_n||z^{\underline{n}}|$ converges, we say that $Y(z)$ in \eqref{1.5} absolutely converges at $z$. We write  the limit function of $Y(z)$ as $y(z)$.
Let $\{\alpha_n\}$ be a sequence satisfying $|\alpha_n|\to 0$. We define a quantity concerning  $\{\alpha_n\}$ as
\begin{equation}
\chi(\{\alpha_n\})=\limsup_{n\to\infty}\frac{n\log n}{-\log|\alpha_n|}.\label{3.1}
\end{equation}

 In \cite{IW2021} we consider properties of binomial series in the complex domain and discuss a criterion for convergence of binomial series in connection with the order of growth of entire functions. We state our result as follows.
\begin{lettertheorem} \cite[Theorem 1.1]{IW2021}\label{Convergence}
Suppose that $\chi(\{a_n\})<1$. Then the formal series $Y(z)$ given by \eqref{1.5} converges to $y(z)$ uniformly on every compact subset in $\mathbb C$. Moreover, the order of growth of $y(z)$ coincides with $\chi(\{a_n\})$.
\end{lettertheorem}

We define the \emph{difference power function} as
    \begin{equation}\label{differencepower.eq}
    z^{\underline{\rho}}=\frac{\Gamma(z+1)}{\Gamma(z+1-\rho)}
    \end{equation}
for any $\rho\in\C$, which yields
    \begin{equation*}
    \begin{split}
    \Delta z^{\underline{\rho}}&=\frac{\Gamma(z+2)}{\Gamma(z+2-\rho)}-
    \frac{\Gamma(z+1)}{\Gamma(z+1-\rho)}
    =\left(\frac{z+1}{z+1-\rho}-1\right)\frac{\Gamma(z+1)}{\Gamma(z+1-\rho)}\\
    &=\rho\frac{\Gamma(z+1)}{(z+1-\rho)\Gamma(z+1-\rho)}
    =\rho z^{\underline{\rho-1}}.
    \end{split}
    \end{equation*}
When $\rho\in\N^+$, it is the definition of the falling factorial \eqref{1.3}. Consider the formal solution of \eqref{linear.eq} of the form
    \begin{equation}\label{frho.eq}
    f(z)=\sum_{n=0}^\infty a_nz^{\U{n+\rho}},
    \end{equation}
where $\rho\in\C$ and $a_n\in\C$ for $n=0,1,\ldots$. We note that
    $$
    z^{\U{n+\rho}}=\frac{\Gamma(z+1)}{\Gamma(z+1-n-\rho)}
    =\frac{\Gamma(z+1)}{\Gamma(z+1-\rho)}(z-\rho)^{\U{n}}=z^{\U{\rho}}(z-\rho)^{\U{n}}.
    $$
Therefore, the series \eqref{1.5} converges if and only if
the series \eqref{frho.eq} converges by
    $$
    f(z)=\sum_{n=0}^\infty a_nz^{\U{n+\rho}}=
    \sum_{n=0}^\infty a_nz^{\U{\rho}}(z-\rho)^{\U{n}}=
    z^{\U{\rho}}\sum_{n=0}^\infty a_n (z-\rho)^{\U{n}}.
    $$

The following lemma is about the product of falling factorial and difference power function.

\begin{lemma}\label{one.lemma}
For any $m\in\N$ and $\rho\in\C$, we have
    $$
    z^{\underline{m}}z^{\underline{\rho}}=\sum_{j=0}^m\binom{m}{j}
    \rho^{\underline{j}}z^{\underline{\rho+m-j}}.
    $$
\end{lemma}

\begin{proof}
We prove this lemma by induction. When $m=1$, it is
    $$
    zz^{\underline{\rho}}=(z-\rho)z^{\underline{\rho}}+\rho z^{\U{\rho}}=(z-\rho)\frac{\Gamma(z+1)}{\Gamma(z+1-\rho)}
    +\rho z^{\U{\rho}}=z^{\U{\rho+1}}+\rho z^{\U{\rho}}.
    $$
We assume that the assertion is true when $m=k$. Keeping in mind that
    $$
    \binom{p}{q}+\binom{p}{q-1}=\binom{p+1}{q}
    $$
for any $p,q\in\N^+$ and $p\leq q$, which is called Pascal's rule. If $m=k+1$, then
    \begin{equation*}
    \begin{split}
    z^{\U{k+1}}z^{\U{\rho}}&=(z-k)\sum_{j=0}^k\binom{k}{j}\rho^{\U{j}}
    z^{\U{\rho+k-j}}=\sum_{j=0}^k\binom{k}{j}\rho^{\U{j}}\left(z^{\U{\rho+k-j+1}}
    +(\rho-j)z^{\rho+k-j}\right)\\
    &=\sum_{j=0}^k\binom{k}{j}\rho^{\U{j}}z^{\U{\rho+k-j+1}}
    +\sum_{j=0}^k\binom{k}{j}\rho^{\U{j+1}}z^{\U{\rho+k-j}}\\
    &=z^{\U{\rho+k+1}}+\sum_{j=1}^k\binom{k}{j}\rho^{\U{j}}z^{\U{\rho+k-j+1}}
    +\sum_{j=1}^k\binom{k}{j-1}\rho^{\U{j}}z^{\U{\rho+k-j+1}}
    +\rho^{\U{k+1}}z^{\rho}\\
    &=z^{\U{\rho+k+1}}+\sum_{j=1}^k\binom{k+1}{j}\rho^{\U{j}}z^{\U{\rho+k-j+1}}
     +\rho^{\U{k+1}}z^{\rho}\\
     &=\sum_{j=0}^{k+1}\binom{k+1}{j}\rho^{\U{j}}z^{\U{\rho+k-j+1}}.
    \end{split}
    \end{equation*}
Therefore, we prove our assertion.
\end{proof}

\section{Asymptotic behavior of solutions of linear difference equations}\label{ams.sec}

In this section, we recall the fundamental theorems on the
asymptotic behavior of solutions of linear difference equations. It is widely accepted among researchers in difference equations that the theorem of Poincar\'e \cite{Poincare} makes the beginning of research in the qualitative theory of linear difference equations, see \cite[Section 8.2]{Elaydi-2005} or \cite[Section 5.3]{KP2001}. In 1885 the French mathematician H. Poincar\'e studied the linear difference equations of the form
    \begin{equation}\label{Poincare.eq}
    x(n+k)+p_1(n)x(n+k-1)+\cdots+p_k(n)x(n)=0
    \end{equation}
such that there are real numbers $p_i$, $1\leq i\leq k$, $k\in\N^+$, with
    \begin{equation}\label{Poincare condition}
    \lim_{n\to\infty} p_i(n)=p_i, \quad\quad 1\leq i\leq k.
    \end{equation}
An equation of the form \eqref{Poincare.eq} with the condition \eqref{Poincare condition} is called a difference equation of \emph{Poincar\'e type}. The characteristic equation associated with \eqref{Poincare.eq} is
    \begin{equation}\label{PoinCh.eq}
    \lambda^k+p_1\lambda^{k-1}+\cdots+p_k=0.
    \end{equation}

\bigskip
\noindent
\textbf{Poincar\'e's Theorem}. \emph{Suppose that condition \eqref{Poincare condition} holds and the characteristic roots $\lambda_1,\lambda_2,\ldots,\lambda_k$ of \eqref{PoinCh.eq} have distinct modulus. If $x(n)$ is a solution of \eqref{Poincare.eq}, then either $x(n)=0$ for all large $n$ or
    \begin{equation}\label{asy.eq}
    \lim_{n\to\infty}\frac{x(n+1)}{x(n)}=\lambda_i
    \end{equation}
for some $i$, $1\leq i\leq k$.}

\bigskip

Note that Poincar\'s Theorem does not tell us whether or not each characteristic root $\lambda_i$ can be written in the form \eqref{asy.eq}. In 1909 O. Perron in \cite{Perron1909} gave an affirmative answer to this question.

\bigskip
\noindent
\textbf{Perron's First Theorem}. \emph{Suppose that $p_k(n)\neq 0$ for all $n\in\N^+$ and the assumptions of Poincar\'e's Theorem hold. Then a fundamental set of solutions $\{x_1(n),x_2(n),\ldots,x_k(n)\}$ of \eqref{Poincare.eq} satisfies the property
    $$
    \lim_{n\to\infty}\frac{x_i(n+1)}{x_i(n)}=\lambda_i, \quad 1\leq i\leq k.
    $$}

Later Perron in \cite{Perron1921} proved an asymptotic result of a different type, which is true without any restriction on the roots of characteristic equation \eqref{PoinCh.eq}.

\bigskip
\noindent
\textbf{Perron's Second Theorem} \emph{Suppose that $p_k(n)\neq 0$ for all $n\in\N^+$ and \eqref{PoinCh.eq} holds. Then \eqref{Poincare condition} has a fundamental set of solutions $\{x_1(n),x_2(n),\ldots,x_k(n)\}$ of \eqref{Poincare.eq} with the property
    $$
    \limsup_{n\to\infty}\sqrt[n]{|x_i(n)|}=\lambda_i, \quad 1\leq i\leq k.
    $$}

In 2002 Pituk \cite{Pituk} obtained a result, similarly to the situation in Poincar\'e theorem, without assumptions on $p_k(n)\neq 0$.

\bigskip
\noindent
\textbf{Pituk's Theorem} \emph{Suppose the condition \eqref{Poincare condition} holds. If $x(n)$ is a solution of \eqref{Poincare.eq}, then either $x(n)=0$ for all large $n$ or
    $$
    \limsup_{n\to\infty}\sqrt[n]{|x(n)|}=\lambda
    $$
is equal to the modulus of one of the roots of the characteristic equation \eqref{PoinCh.eq}.}

\bigskip

The remarkable paper on the asymptotic behaviour of solutions is due to C.R. Adams \cite{Adams}. We obtain the main term of formal series solutions by the results in his paper. We focus on the asymptotic behaviour of solutions of linear difference equations
    \begin{equation}\label{LDE.eq}
    p_0(n)x(n+k)+p_1(n)x(n+k-1)+\cdots+p_k(n)x(n)=0
    \end{equation}
whose coefficient functions are expressible in the form
    $$
    p_i(n)=a_{i0}+a_{i1}n^{-1}+a_{i2}n^{-2}+\cdots
    $$
for all large $n$ and whose characteristic equation is
    \begin{equation}\label{LEDch.eq}
    a_{00}\lambda^k+a_{10}\lambda^{k-1}+\cdots+a_{k-1,0}\lambda+a_{k0}=0.
    \end{equation}
It is known that there are $k$ linearly independent solutions of \eqref{LDE.eq}. Let us just summarize Adams' results on calculating main terms of formal solutions as follows, which is enough for us to consider the behaviour of binomial series in Section 5.

\textbf{Case 1 ($a_{00}\neq 0$ and $a_{k0}\neq 0$)}. Corresponding to a root $\lambda$ of \eqref{LEDch.eq} with multiplicity $m\geq 1$, there are $m$ linearly independent formal series of the following type
    \begin{equation}\label{regular.eq}
    x(n)\sim\lambda^n n^re^{L(n)},
    \end{equation}
where $r$ is a constant, $L(n)$ is a polynomial of $n^{1/j}$
with coefficients being different in the different series of degree at most $j-1$ for some $j\leq m$. Obviously, there are $k$ linearly independent formal solutions of the form \eqref{regular.eq} corresponding to the characteristic equations
\eqref{LEDch.eq}.

\textbf{Case 2 (one or both $a_{00},a_{k0}=0$)}. Let us denote by $a_{i,j_i}$ the first nonzero coefficients in $p_i(n)$ ($i=0,1,\ldots,k$), and choosing $i-$ and $j-$ axes, plot the points $(i,j_i)$ as in Figure 1. Construct broken lines $L$, convex upward, such that both ends of each segment of the line are points of the set $(i,j_i)$ and such that all points of the set lie upon or beneath the line, see e.g., \cite[P. 511]{Adams}. Let the slope of any such segment $L$ be $\mu$, and the number of points $(i,j_i)$ that lie on or beneath that segment $L$ is $\nu$. Then the degree of the characteristic equation associated with that segment $L$ is $\nu-1$. There are $\nu-1$ linearly independent formal solutions associated with the segment $L$ of the following type
    \begin{equation}\label{irregular}
    x(n)\sim n^{-\mu n}e^{\mu n}\lambda^n e^{L(n)}n^r,
    \end{equation}
where $\lambda$ is a root of the characteristic equation associated with that segment $L$, $r$ is a constant, and $L(n)$ is a polynomial of $n^{1/j}$
with coefficients being different in the different series of degree at most $j-1$ for some $j\leq \nu-1$. Evidently the sum
of the degrees of these several characteristic equations is $k$. Obviously, there are $k$ linearly independent formal solutions of the form \eqref{irregular} of difference equations \eqref{LDE.eq}.

\begin{figure}[H]\label{Line.fi}
 \begin{center}
    \begin{tikzpicture}[scale=0.8]
    \draw[->](-2,0)--(13,0)node[left,below]{$i$};
    \draw[->](0,1)--(0,-8)node[left]{$j$};
    \draw[-](0,0)--(0,0)node[left=6pt,above]{$0$};
    \draw[-,thick](0,-7)--(1,-4)node[left,below=8pt]{$L$};
    \draw[-,thick](0,-7)--(0,-7)node[below,right]{\small{$(0,j_0)$}};
    \draw[-,thick](1,-4)--(2,-2);
    \draw[-,thick](2,-2)--(4,0);
    \draw[thick](0,-7)node{\blue{$\bullet$}};
    \draw[thick](1,-4)node{\blue{$\bullet$}};
    \draw[thick](2,-2)node{\blue{$\bullet$}};
    \draw[thick](4,0)node{\blue{$\bullet$}};
    \draw[thick](4,0)node{\blue{$\bullet$}};
    \draw[thick](3,-3)node{\blue{$\bullet$}};
    \draw[thick](5,-2)node{\blue$\bullet$};
    \draw[thick](6,0)node{\blue{$\bullet$}};
    \draw[thick](7,0)node{\blue{$\bullet$}};
    \draw[thick](8,-1)node{\blue{$\bullet$}};
    \draw[thick](9,-3)node{\blue{$\bullet$}};
    \draw[thick](10,-6)node{\blue{$\bullet$}};
    \draw[draw=red,dashed](1,-4)--(1,0);
    \draw[draw=red,dashed](2,-2)--(2,0);
    \draw[draw=red,dashed](8,-1)--(8,0);
    \draw[draw=red,dashed](9,-3)--(9,0);     \draw[draw=red,dashed](10,-6)--(10,0);
    \draw[dashed](3,-3)--(2,-2);
    \draw[dashed](3,-3)--(4,0);
    \draw[dashed](4,0)--(5,-2);
    \draw[dashed](5,-2)--(7,0);
    \draw[-,thick](4,0)--(7,0);
    \draw[-,thick](7,0)--(8,-1);
    \draw[-,thick](8,-1)--(9,-3);
    \draw[-,thick](9,-3)--(10,-6)node[below,right]{\small{$(k,j_k)$}};
    \end{tikzpicture}
    \end{center}
	\begin{quote}
    \caption{Convex curve of linear difference equations}
	\end{quote}
    \end{figure}
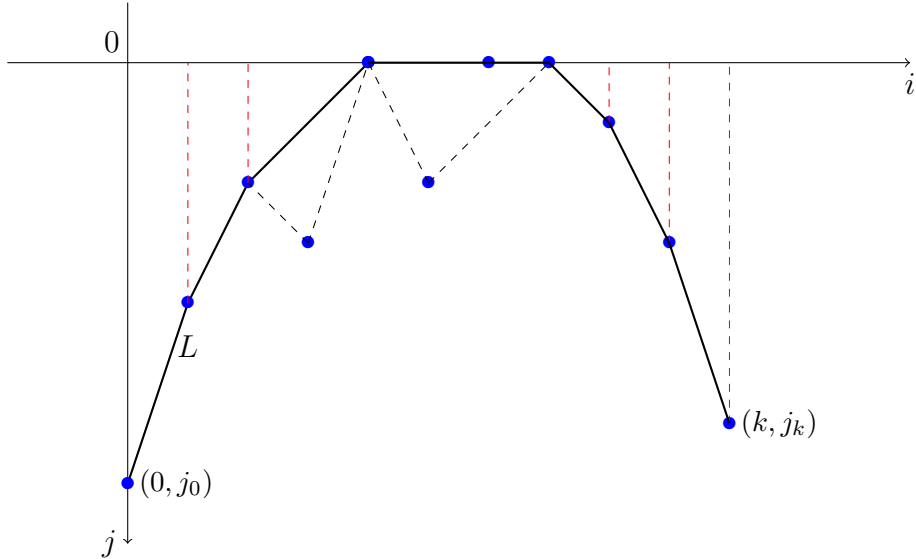

We note here that asymptotic behaviour of solutions of difference equations \eqref{Poincare.eq} of the form \eqref{regular.eq} or \eqref{irregular} is enough for us in this paper, but it is not enough for asymptotic analysis.
In fact, there are lots of paper to show how to calculate
 coefficients of formal series solutions. We refer to \cite{Adams}, \cite{Birkhoff1911, Birkhoff1930}, \cite{BT1933} for more details, and \cite{Wong-Li} for second order case.

\section{Proof of main theorems}

We assume that there exists a formal solution of \eqref{linear.eq} of the following form
    $$
    f(z)=\sum_{n=0}^\infty a_nz^{\U{n}}.
    $$
Obviously, for any given $j>0$, we have
    $$
    \Delta^j f(z)=\sum_{n=0}^\infty a_{n+j}(n+j)^{\U{j}}z^{\U{n}}.
    $$
For any given $j$, $j=0,1,\ldots,m$, we suppose that the polynomial $P_j$ is of the form
    $$
    P_j(z)=A_{j,d_j}z^{\U{d_j}}+A_{j,d_j-1}z^{\U{d_j-1}}
    +\cdots+A_{j,1}z+A_{j,0},
    $$
where $A_{j,i}\in\C$ for $0\leq i\leq d_j$, and $A_{j,d_j}\neq 0$. Therefore, it follows from Lemma \ref{one.lemma} that
    \begin{align*}
    P_j(z)\Delta^j f(z)&=\left(\sum_{t=0}^{d_j}A_{j,t}z^{\U{t}}\right)\left(\sum_{n=0}^\infty a_{n+j}(n+j)^{\U{j}}z^{\U{n}}\right)\\
    &=\sum_{n=0}^\infty \sum_{t=0}^{d_j}\sum_{k=0}^t
    a_{n+j}A_{j,t}\binom{t}{k}(n+j)^{\U{k+j}}z^{\U{n+t-k}}.
    \end{align*}
Hence, it yields from \eqref{linear.eq} that
    \begin{equation}\label{sum.eq}
    \sum_{j=0}^m\sum_{n=0}^\infty \sum_{t=0}^{d_j}\sum_{k=0}^t
    a_{n+j}A_{j,t}\binom{t}{k}(n+j)^{\U{k+j}}z^{\U{n+t-k}}=0.
    \end{equation}
We set $i=t-k$ in \eqref{sum.eq} and obtain that
    \begin{align*}
    &\sum_{j=0}^m\sum_{n=0}^\infty a_{n+j} (n+j)^{\U{j}} \sum_{t=0}^{d_j}\sum_{k=0}^t A_{j,t}\binom{t}{k}n^{\U{k}}z^{\U{n+t-k}}\\
    =&\sum_{j=0}^m\sum_{n=0}^\infty a_{n+j}(n+j)^{\U{j}}
    \sum_{t=0}^{d_j}\sum_{i=0}^tA_{j,t}
    \frac{\Delta^i(n^{\U{t}})}{i!}z^{\U{n+i}}\\
    =&\sum_{j=0}^m\sum_{n=0}^\infty a_{n+j}(n+j)^{\U{j}}
    \sum_{i=0}^{d_j}\sum_{t=i}^{d_j}
    \frac{\Delta^i(A_{j,t} n^{\U{t}})}{i!}z^{\U{n+i}}\\
    =&\sum_{j=0}^m\sum_{n=0}^\infty a_{n+j}(n+j)^{\U{j}}
    \sum_{i=0}^{d_j}\frac{\Delta^i(P_j(n))}{i!}z^{\U{n+i}}=0.
    \end{align*}
Now, we assume that $d=\max_{0\leq j\leq m}\{d_j\}$, and reduce that
    \begin{align*}
    &\sum_{j=0}^m\sum_{n=0}^\infty\sum_{i=0}^{d}
    a_{n+j}(n+j)^{\U{j}}\frac{\Delta^i(P_j(n))}{i!}z^{\U{n+i}}\\
    =&\sum_{j=0}^m\sum_{i=0}^{d}\sum_{n=0}^\infty a_{n+j}(n+j)^{\U{j}}\frac{\Delta^i(P_j(n))}{i!}z^{\U{n+i}}\\
    =&\sum_{j=0}^m\sum_{i=0}^{d}\sum_{n=i}^\infty
    a_{n-i+j}(n-i+j)^{\U{j}}\frac{\Delta^i(P_j(n-i))}{i!}z^{\U{n}}\\
    =&\sum_{j=0}^m\left(\sum_{n=0}^{d-1}\sum_{i=0}^n+
    \sum_{n=d}^\infty\sum_{i=0}^d\right)
    a_{n-i+j}(n-i+j)^{\U{j}}\frac{\Delta^i(P_j(n-i))}{i!}z^{\U{n}}\\
    =&\left(\sum_{n=0}^{d-1}\sum_{i=0}^n+    \sum_{n=d}^\infty\sum_{i=0}^d\right)\sum_{j=0}^m
     a_{n-i+j}(n-i+j)^{\U{j}}\frac{\Delta^i(P_j(n-i))}{i!}z^{\U{n}}=0.
    \end{align*}
Since $z^{\U{n}}$ are linearly independent over the periodic field with period one for distinct $n$, thus
    \begin{equation}\label{np1.eq}
    \sum_{i=0}^n\sum_{j=0}^m
     a_{n-i+j}(n-i+j)^{\U{j}}\frac{\Delta^i(P_j(n-i))}{i!}=0
     \quad\text{for}\quad n< d,
     \end{equation}
and
    \begin{equation}\label{np2.eq}
    \sum_{i=0}^d\sum_{j=0}^m
     a_{n-i+j}(n-i+j)^{\U{j}}\frac{\Delta^i(P_j(n-i))}{i!}=0
     \quad\text{for}\quad n\geq d.
     \end{equation}
We consider the asymptotic behaviour of $a_n$ for large $n$ from \eqref{np2.eq} and ignore \eqref{np1.eq}. We assume that $\Delta^\alpha f(z)=0$ for $\alpha<0$. Therefore, we write \eqref{np2.eq} as
    \begin{align*}
   &\sum_{i=0}^d\sum_{j=0}^ma_{n-i+j}(n-i+j)^{\U{j}}\frac{\Delta^i(P_j(n-i))}{i!}\\
   =&\sum_{j=0}^m\sum_{i=0}^da_{n-i+j}(n-i+j)^{\U{j}}\frac{\Delta^i(P_j(n-i))}{i!}\\
   =&\sum_{j=0}^m\sum_{t=-j}^{d-j}a_{n-t}(n-t)^{\U{j}}\frac{\Delta^{t+j}(P_j(n-j-t))}{(t+j)!}\\
   =&\sum_{j=0}^m\sum_{i=-m}^{d}a_{n-i}(n-i)^{\U{j}}\frac{\Delta^{i+j}(P_j(n-j-i))}{(i+j)!}\\
   =&\sum_{i=-m}^{d}a_{n-i}\sum_{j=0}^m(n-i)^{\U{j}}\frac{\Delta^{i+j}(P_j(n-j-i))}{(i+j)!}.
    \end{align*}
Let us set
    $$
    Q(n,i)=\sum_{j=0}^m\frac{(n-i)^{\U{j}}}{(i+j)!}\Delta^{i+j}(P_j(n-j-i))
    $$
for $i=-m,\ldots,d$. We see that $a_n$ satisfies a recurrence relation of $m+d$ order
    \begin{equation}\label{recu.eq}
    a_{n+m}Q(n,-m)+a_{n+m-1}Q(n,-m+1)+\cdots+a_{n-d}Q(n,d)=0.
    \end{equation}
In the following, we proceed to seek for the linearly independent asymptotic solutions of \eqref{recu.eq}.
The proofs of theorems heavily depend on asymptotic behaviour of solutions of linear difference equations.

\subsection{Proof of Theorem \ref{one.theorem}}

Suppose that $p=1$. Since $d_{s_1}\geq d_k$ for
$k\geq s_1$ and $d_{s_1}> d_k$ for
$0\leq k< s_1$, we have
    \begin{equation}\label{Qnk.eq}
    \deg{Q(n,k)}\leq d_{s_{1}}-k~~\quad\text{for}\quad -m\leq k<d_{s_1}-s_{1}
    \end{equation}
and $\deg{Q(n,d_{s_1}-s_1)}=s_1$, and
    $$
    Q(n,k)=0 ~\quad\quad\quad\quad\text{for}\quad d_{s_1}-s_1<k\leq d.
    $$
We rewrite \eqref{recu.eq} as
    \begin{equation}\label{recu2.eq}
     a_{n+m}Q(n,-m)+a_{n+m-1}Q(n,-m+1)+\cdots+a_{n-(d_{s_1}-s_1)}Q(n,d_{s_1}-s_1)=0.
    \end{equation}
 We denote by an index $\xi$ satisfying $-m\leq \xi\leq d$ such that $\deg{Q(n,\xi)}>\deg{Q(n,k)}$ for $\xi>k$ and
$\deg{Q(n,\xi)}\geq\deg{Q(n,k)}$ for $-m\leq k\leq d$.
It is obvious that $\xi\geq -s_1$. By divided $Q(n,\xi)$
in both side of \eqref{recu2.eq}, let us write \eqref{recu2.eq} as
    \begin{equation}\label{recu3.eq}
    a_{n+m}A_0(n)+a_{n+m-1}A_1(n)+\cdots+a_{n-(d_{s_1}-s_1)}A_{m+d_{s_1}-s_1}(n)=0,
    \end{equation}
where
    $$
    A_i(n)=A_{i,0}+A_{i,1}n^{-1}+A_{i,2}n^{-2}+\cdots
    $$
for $i=0,1,\ldots,m+d_{s_1}-s_1$. Let us denote by $A_{i,j_i}$ the first nonzero coefficient in $A_i$
for $i=0,1,\ldots,m+d_{s_1}-s_1$. Choosing $i$ and $j$ axes, we plot the points $(i,j_i)$.
Construct a segment $L$, convex, such that both ends of the line are points of the set $(i,j_i)$ and such that all points of the set lie upon or beneath the line.
Obviously, the number of these segments are not greater than $m+d_{s_1}-s_1$. The segments are denoted by $L_1,L_2,\ldots$, respectively, and slopes of such segments are denoted by $\mu_1,\mu_2,\ldots$, which are rational numbers. The degrees of characteristic equations corresponding to these segments are denoted by $\nu_1,\nu_2,\ldots$. In addition, the point $(0,j_0)$ is the beginning point and the point $(m+d_{s_1}-s_1,j_{m+d_{s_1}-s_1})$ is the ending point from left to right, see Figure 1.

From the construction of the convex lines, we see that
$\mu_t$ is increasing for $t$, and
    $$
    j_i=\deg{Q(n,\xi)}-\deg{Q(n,i-m)}
    $$
for $i=0,1,\ldots,m+d_{s_1}-s_1$. Let us suppose that the last segment is denoted by $L_\chi$ which connects the points
$(t,j_t)$ and $(m+d_{s_1}-s_1,j_{m+d_{s_1}-s_1})$. Therefore, it follows from \eqref{Qnk.eq} that
    \begin{align*}
    \mu_\chi&=\frac{\deg{Q(n,t-m)}-\deg{Q(n,d_{s_1}-s_1)}}{(m+d_{s_1}-s_1)-t}\\
    &\leq \frac{(d_{s_1}-(t-m))-s_1}{(m+d_{s_1}-s_1)-t}\\
    &=1.
    \end{align*}
It implies that all the slopes of segments are no greater than 1. By \cite{Adams}, for any $j=1,2,\ldots,\chi$, we know there are $\nu_j$ linearly independent formal solution of \eqref{recu3.eq} corresponding to the segment $L_j$ of the form
    $$
    a_{n,j}\sim n^{-\mu_j n}e^{\mu_j n}\lambda_j^ne^{L_j(n)}n^{r_j}\quad\quad\text{as}\quad n\to\infty,
    $$
where $\lambda_j$ is a root of the characteristic equation associated with that segment $L_j$, $r_j$ is a constant, and $L_j(n)$ is a polynomial of $n^{1/j}$
with coefficients being different in the different series of degree at most $j-1$ for some $j\leq \nu_j$. Hence, we reduce
    $$
    \chi(\{a_n,j\})=\limsup_{n\to\infty}\frac{n\log n}{-\log|a_{n,j}|}=\frac{1}{\mu_j}\geq 1
    $$
for $j=1,2,\ldots,\chi$. From Theorem~\ref{Convergence} and \cite[Corollary A. 1]{IW2021}, we have $f$ cannot be entire of order less than 1. We prove our assertion. \hfill $\square$

\subsection{Proof of Theorem \ref{two.theorem}}

In order to prove Theorems \ref{two.theorem}, we proceed to seek for all linearly independent asymptotic solutions of \eqref{recu.eq} by the method in \cite{Adams}. Since the relation \eqref{s1.eq} to \eqref{sj2.eq} hold, we have
    \begin{equation}\label{Qdegree.eq}
    \begin{split}
    \deg{Q(n,k)}&\leq d_{s_{1}}-k~~\quad\text{for}\quad -m\leq k<d_{s_1}-s_{1};\\
    \deg{Q(n,k)}&=s_j\quad\quad\quad\quad\text{for}\quad k=d_{s_j}-s_j;\\
    \deg{Q(n,k)}&\leq d_{s_{j+1}}-k\quad\text{for}\quad d_{s_j}-s_j<k<d_{s_{j+1}}-s_{j+1};\\
    Q(n,k)&=0 ~\quad\quad\quad\quad\text{for}\quad d_{s_p}-s_p<k\leq d.
    \end{split}
    \end{equation}
We rewrite \eqref{recu.eq} as
    \begin{equation}\label{recu4.eq}
     a_{n+m}Q(n,-m)+a_{n+m-1}Q(n,-m+1)+\cdots+a_{n-(d_{s_p}-s_p)}Q(n,d_{s_p}-s_p)=0.
    \end{equation}
It is well known that there are $m+d_{s_p}-s_p$ linearly independent solutions of \eqref{recu4.eq}.
Let us set
    $$
   a_{n+m}=\frac{x(n)}{[(n+m)!]^\mu},
    $$
where $\mu\in\R$. Then $x(n)$ satisfies the recurrence relation
    \begin{equation}\label{xrelation.eq}
    x(n)T(n,\mu,0)+x(n-1)T(n,\mu,1)+\cdots+x(n-(d_{s_p}-s_p))T(n,\mu,d_{s_p}-s_p)=0,
    \end{equation}
where
    $$
    T(n,\mu,i)=Q(n,i-m)[(n+m)^{\U{i}}]^\mu
    $$
for $i=0,1,\ldots,m+d_{s_p}-s_p$. The highest power on $n$ of $T(n,\mu,i)$ is denoted by $\deg{T(n,\mu,i)}$. Obviously,
    $$
    \deg{T(n,\mu,i)}=\deg{Q(n,i-m)}+i\mu
    $$
for $i=0,1,\ldots,m+d_{s_p}-s_p$.
Since we assume that $p\geq 2$, for given $j=1,2,\ldots,p-1$ and for any $\mu>1$,we deduce from \eqref{Qdegree.eq} that
    \begin{equation}\label{Tdegree.eq}
    \begin{split}
    &\deg{T(n,\mu,m+d_{s_{j+1}}-s_{j+1})}>\deg{T(n,i)}~~\text{for}\quad m+d_{s_j}-s_j<i<m+d_{s_{j+1}}-s_{j+1};\\
    &\deg{T(n,\mu,m+d_{s_1}-s_1})>\deg{T(n,i)} ~\quad\quad\text{for}\quad 0\leq i<m+d_{s_1}-s_{1}.\\
%    &\deg{T(n,\mu,m+d_{s_{p}}-s_{p})}>\deg{T(n,i)} ~\quad\quad\text{for}\quad m+d_{s_p}-s_p<k\leq m+d.
    \end{split}
    \end{equation}
Now let us choose $\mu_j=1/\rho_j$ for $j=1,2,\ldots,p-1$, we have
    \begin{equation}\label{Teq.eq}
    \deg{T(n,\mu_j,m+d_{s_j}-s_j)}=\deg{T(n,\mu_j,m+d_{s_{j+1}}-s_{j+1})}
    \end{equation}
for any $j=1,2,\ldots,p-1$. Moreover, if $p\geq 3$, it yields that for any given $j=1,2,\ldots,p-1$
    \begin{equation}\label{Tbig.eq}
    \begin{split}
    &\deg{T(n,\mu_j,m+d_{s_{j}}-s_{j})}-\deg{T(n,\mu_j,m+d_{s_k}-s_k)}\\
    =&s_{j}+(m+d_{s_{j}}-s_{j})\mu_j-s_k-(m+d_{s_k}-s_k)\mu_j\\
    =&(s_j-s_k)+(d_{s_j}-s_j+s_k-d_{s_k})\mu_j\\
    =&(s_j-s_k)\left(1-\left(\frac{d_{s_j}-d_{s_k}}{s_k-s_j}+1\right)\mu_j\right)\\
    >&(s_j-s_k)\left(1-\left(\frac{d_{s_j}-d_{s_{j+1}}}{s_{j+1}-s_j}+1\right)\mu_j\right)\\
    =&(s_j-s_k)(1-\mu_j/\mu_{j})\\
    =&0
    \end{split}
    \end{equation}
when $j+1<k\leq p$. Thus, from \eqref{Tdegree.eq} to
\eqref{Tbig.eq}, we deduce that for a given $j=1,2,\ldots,p-1$,
    $$
     \deg{T(n,\mu_j,m+d_{s_j}-s_j)}=\deg{T(n,\mu_j,m+d_{s_{j+1}}-s_{j+1})}
     >\deg{T(n,i)}
    $$
holds for $i\neq m+d_{s_j}-s_j$ and $i\neq m+d_{s_{j+1}}-s_{j+1}$. According to \cite{Adams}, the characteristic equation for $\mu_j$ is
    \begin{equation}\label{ch.eq}
    A\gamma^{m+d_{s_j}-s_j}+B\gamma^{m+d_{s_{j+1}}-s_{j+1}}=0,
    \end{equation}
where $A$ and $B$ are constants. There exist $(d_{s_{j+1}}-s_{j+1})-(d_{s_j}-s_j)$ nonzero simple roots of equation \eqref{ch.eq}, which are denoted by $\gamma_{j,t}$ for $t=1,2,\ldots,(d_{s_{j+1}}-s_{j+1})-(d_{s_j}-s_j)$.
Then for a give $j=1,2,\ldots, p-1$, we find at most $(d_{s_{j+1}}-s_{j+1})-(d_{s_j}-s_j)$ linearly independent solutions of \eqref{recu4.eq} of asymptotic behaviour as
    $$
    a_n^{(j,t)}\sim n^{-\mu_jn}e^{\mu_jn}\gamma_{j,t}^ne^{L_j(n)}n^{r_j},
    $$
where $r_j$ are constants and $L_j(n)$ are polynomials in $n^{1/j}$. Moreover, for a given $j=1,2,\ldots,p-1$, we have
    $$
    \chi(\{a_n^{(j,t)}\})=\limsup_{n\to\infty}\frac{n\log n}{-\log|a_n^{(j,t)}|}=\frac{1}{\mu_j}< 1.
    $$
Form Theorem~\ref{Convergence}, for a give $j=1,2,\ldots, p-1$, we find at most $(d_{s_{j+1}}-s_{j+1})-(d_{s_j}-s_j)$ linearly independent entire solutions of order $\rho_j=1/\mu_j<1$ of \eqref{linear.eq}.

We mention here that there exist at most $\sum_{j=1}^{p-1}[(d_{s_{j+1}}-s_{j+1}))-(d_{s_j}-s_j)]$ transcendental entire solutions of order less than one under our assumption. By (2.1) to (2.4), we have
$$
\sum_{j=1}^{p-1} [(d_{s_{j+1}}-s_{j+1}))-(d_{s_j}-s_j)]=(s_1-s_p)-(d_{s_1}-d_{s_p})<m.
$$
Hence other $m-[(s_1-s_p)-(d_{s_1}-d_{s_p})]$ linearly independent solutions over periodic field with period one of \eqref{linear.eq} can not be entire of order less than 1.
The method and tools are similar as the proof of Theorem 2.1, we omit the details here. Therefore, we prove our assertion of Theorem 2.2. \hfill $\square$

\section{Entire solution of rational order between 0 and 1}

In this section, we will construct a linear difference equation \eqref{linear.eq} with polynomial coefficients, which has an entire solution of rational order between 0 and 1.

\begin{theorem}\label{construct.theorem}
For any positive rational number $\lambda<1$, there exists
a linear difference equation \eqref{linear.eq} with polynomial coefficients such that $f$ is an entire solution of \eqref{linear.eq} of order $\rho(f)=\lambda$.
\end{theorem}

\begin{proof}
For any given positive rational number $\lambda<1$, we write
    $$
    \lambda=\frac{q}{p},
    $$
where $p$ and $q$ are relatively prime positive integer such that $q<p$.
We set a linear difference equation
    \begin{equation}\label{pq.eq}
    A_pz^{\underline{p}}\Delta^p f(z-p)+\cdots+
    A_1z\Delta f(z-1)-A_0z^{\underline{q}}f(z-q)=0,
    \end{equation}
where $A_j$ are constants for $j=0,\ldots,p$.
We consider a formal solution $f(z)=\sum_{n=0}^\infty a_nz^{\underline{n}}$ given by \eqref{frho.eq}. Since
    $$
    z^{\underline{k}}\Delta^m\left(\sum_{n=0}^\infty a_n(z-k)^{\underline{n}}\right)
    =\sum_{n=0}^\infty a_nn^{\underline{m}}z^{\underline{n-m+k}}
    $$
for any positive integer $k$ and $m$, we write \eqref{pq.eq} as
    \begin{equation}\label{Apq.eq}
    A_p\sum_{n=0}^\infty a_nn^{\underline{p}}
    z^{\underline{n}}+\cdots+
    A_1\sum_{n=0}^\infty a_nnz^{\underline{n}}
    -A_0\sum_{n=q}^\infty a_{n-q}z^{\underline{n}}=0.
    \end{equation}
Hence, it gives us that $a_1=\cdots=a_{q-1}=0$ and
    $$
    f(n)a_{n}=a_{n-q}
    $$
for $n\geq q$, where $f(n)=(A_pn^{\underline{p}}
+\cdots+A_1n)/A_0$. Hence we see that $a_{qt-1}=\cdots=a_{qt-q+1}=0$ for $t\in\N$ by setting $n=qt$, and get
    \begin{equation}\label{aqt.eq}
    f(qt)a_{qt}=a_{q(t-1)}.
    \end{equation}
Now let us set $f(qt)=(pt)^{\underline{p}}$, namely,
we choose $A_0,\ldots,A_p$ such that
    \begin{equation}\label{fpq.eq}
    \frac{A_pn^{\underline{p}}
+\cdots+A_1n}{A_0}=\left(\frac{n}{\lambda}\right)^{\underline{p}}.
    \end{equation}
It is easy to see that $A_0/A_p=\lambda^p$ from \eqref{fpq.eq}. It follows from \eqref{aqt.eq} that
    $$
    a_{qt}=\frac{a_{q(t-1)}}{(pt)^{\underline{p}}}=\frac{a_{q(t-2)}}{(pt)^{\underline{p}}(p(t-1))^{\underline{p}}}=\cdots=\frac{a_0}{(pt)!}.
    $$
It implies that $f$ is of the form
    $$
    f(z)=\sum_{t=0}^\infty \frac{a_0}{(pt)!}z^{\underline{qt}}.
    $$
By means of Theorem~\ref{Convergence}, the formal solution converges to an entire function of order $\lambda=q/p$, since
    \begin{equation}\label{lambda.eq}
    \chi(\{a_{qt}\})=\limsup_{t\to\infty}\frac{qt\log(qt)}{-\log(pt)!}=\frac{q}{p}=\lambda<1.
    \end{equation}

We have thus proved that \eqref{pq.eq} possesses an entire solution of order $\lambda$.
We set $z+q$ in place of $z$ in \eqref{pq.eq}, and use a formula
\begin{equation}\label{summer.eq}
\Delta^m f(z+k)= \sum_{j=0}^m \frac{m!}{j!(m-j)!}(-1)^j\sum_{i=0}^{k+m-j}\frac{(k+m-j)!}{i!(k+m-j-i)!}\Delta^i f(z),
 \end{equation}
for non-negative integers $m$ and $k$. Then we obtain a difference equation of the form \eqref{linear.eq}.
\end{proof}

\section{Examples}

In this section we give several examples which illustrate our theorems.

\begin{example}\label{one.example}
In \cite[Remark 6.3]{IY2004} Ishizaki and Yanagihara shows that there exists a transcendental entire solution $f$ of order $\rho(f)=1/3$ of the difference equation
    \begin{equation}\label{one.ex}
    (6z^2+19z+15)\Delta^3 f(z)+(z+3)\Delta^2 f(z)-\Delta f(z)-f(z)=0.
    \end{equation}
From Theorems \ref{one.theorem} to \ref{two.theorem},
we know that $p=2$, $s_1=3$ and $s_2=0$. Hence, there
exists at least one transcendental entire solution of order $\rho(f)<1$ of \eqref{one.ex}.
In addition, the only possibility of order is $\rho_1=1/3$,
and the number of entire solution of order $\rho(f)=1/3$ is one.

Moreover, suppose that \eqref{one.ex} has a formal solution
    $$
    f(z)=\sum_{n=0}^\infty a_nz^{\U{n+\rho}}.
    $$
We substitute it into \eqref{one.ex}, and by Lemma \ref{one.lemma} we obtain that
    \begin{equation*}
    \begin{split}
    6&\sum_{n=-1}^\infty (n+1+\rho)^{\U{3}} a_{n+1}z^{\U{n+\rho}}
    +12\sum_{n=-2}^\infty (n+2+\rho)^{\U{4}} a_{n+2}z^{\U{n+\rho}}+6\sum_{n=-3}^\infty (n+3+\rho)^{\U{5}}
    a_{n+3}z^{\U{z+\rho}}\\
    +&25\sum_{n=-2}^\infty(n+2+\rho)^{\U{3}}a_{n+2}z^{\U{n+\rho}}
    +25\sum_{n=-3}^\infty(n+3+\rho)^{\U{4}}a_{n+3}z^{\U{n+\rho}}\\
    +&15\sum_{n=-3}^\infty(n+3+\rho)^{\U{3}}a_{n+3}z^{\U{n+\rho}}
    +\sum_{n=-1}^\infty(n+1+\rho)^{\U{2}}a_{n+1}z^{\U{n+\rho}}
    +\sum_{n=-2}^\infty(n+2+\rho)^{\U{3}}a_{n+2}z^{\U{n+\rho}}\\
    +&3\sum_{n=-2}^\infty(n+2+\rho)^{\U{2}}a_{n+2}z^{\U{n+\rho}}
    -\sum_{n=-1}^\infty(n+1+\rho)a_{n+1}z^{\U{n+\rho}}
    -\sum_{n=0}^\infty(n+\rho)a_{n}z^{\U{n+\rho}}=0.
    \end{split}
    \end{equation*}
Denoting that
    \begin{equation*}
    \begin{split}
    Q_1(n,\rho)&=6(n+\rho)^{\U{5}}+25(n+\rho)^{\U{4}}+15(n+\rho)^{\U{3}};\\
    Q_2(n,\rho)&=12(n+\rho)^{\U{4}}+26(n+\rho)^{\U{3}}+3(n+\rho)^{\U{2}};\\
    Q_3(n,\rho)&=6(n+\rho)^{\U{3}}+(n+\rho)^{\U{2}}-(n+\rho),
    \end{split}
    \end{equation*}
we have
    \begin{equation}\label{recurrence_ex.eq}
    -a_n+Q_3(n+1,\rho)a_{n+1}+Q_2(n+2,\rho)a_{n+2}+Q_1(n+3,\rho)a_{n+3}=0\quad\text{for}\quad n\geq -3,
    \end{equation}
where $a_j=0$ when $j<0$. Therefore, it follows that $\rho_1=0$, $\rho_2=1$, $\rho_3=2$, $\rho_4=4/3$ and $\rho_5=3/2$ if $n=-3$ in \eqref{recurrence_ex.eq}. We denote
    \begin{equation*}
    \begin{split}
    A_n(\rho)=&(n+\rho+2)(n+\rho+1)(2n+2\rho+1)(3n+3\rho+2)a_{n+2}\\
    &+(n+\rho+1)(n+\rho)(6n+6\rho-5)a_{n+1}-a_n
    \end{split}
    \end{equation*}
for $n\geq 0$, then from \eqref{recurrence_ex.eq}
    $$
    (n+\rho+1)A_{n+1}(\rho)+A_n(\rho)=0.
    $$
If we set $A_0(\rho)=0$, then $A_n=0$ for any $n\geq 0$.
Therefore, it yields
    $$
    (n+\rho+2)(n+\rho+1)(2n+2\rho+1)(3n+3\rho+2)a_{n+2}
    +(n+\rho+1)(n+\rho)(6n+6\rho-5)a_{n+1}-a_n=0.
    $$
Now let us denote
    $$
    B_n(\rho)=(n+\rho+1)(2n+2\rho-1)(3n+3\rho-1)a_{n+1}-a_n,
    $$
thus,
    $$
    (n+\rho+1)B_{n+1}(\rho)+B_n(\rho)=0.
    $$
Let us set $B_n(\rho)=0$, it gives us that $B_n(\rho)=0$
for any $n\geq 0$. Hence, it shows us that
    $$
    (n+\rho+1)(2n+2\rho-1)(3n+3\rho-1)a_{n+1}-a_n=0
    $$
for $n\geq 0$. By the summation formula, see e.g., \cite[P.~48]{KP2001}, we obtain that
    $$
    a_n=\frac{\Gamma(\rho-1/2)\Gamma(\rho-1/3)a_0}{6^n\Gamma(n+\rho+1)\Gamma(n+\rho-1/2)\Gamma(n+\rho-1/3)}.
    $$
The condition $A_0(\rho)=B_0(\rho)=0$ and $n=-2$ in \eqref{recurrence_ex.eq} gives us $\rho=0$, $\rho=3/2$ and $\rho=4/3$.
 In this way, we find three linear independent solutions as follows:
    $$
    f_1(z)=\sum_{n=0}^\infty \frac{1}{6^n\Gamma(n+1)\Gamma(n-1/2)\Gamma(n-1/3)}z^{\U{n}}
    $$
is entire of order $1/3$ from Theorem \ref{Convergence},
    $$
    f_2(z)=\frac{\Gamma(z+1)}{\Gamma(z-1/2)}\sum_{n=0}^\infty \frac{1}{6^n\Gamma(n+1)\Gamma(n-1/2)\Gamma(n-1/3)}(z-3/2)^{\U{n}}
    $$
is meromorphic of order 1 with infinitely many poles,
    $$
    f_3(z)=\frac{\Gamma(z+1)}{\Gamma(z-1/3)}\sum_{n=0}^\infty \frac{1}{6^n\Gamma(n+1)\Gamma(n-1/2)\Gamma(n-1/3)}(z-4/3)^{\U{n}}
    $$
is meromorphic of order 1 with infinitely many poles. We claim that $f_1$, $f_2$ and $f_3$ are linear independent over period field with period one.
\end{example}

\begin{example}
By means of the method in the proof of Theorem~\ref{construct.theorem}, we construct a difference equation that possesses an entire solution of order $3/4$. We now consider a forth order linear difference equation with polynomial coefficient
\begin{multline}
(256 z^3+1920z^2+4656z+3640)\Delta^4 y(z)+(384 z^2+1760z+1944)\Delta^3 y(z)\\
-(80z+120)\Delta^2 y(z)-(81 z^2+ 405 z +446)\Delta y(z)-(81z^2+405z+486)y(z)=0.\label{7.1}
\end{multline}
Let $k\in\mathbb{N}$. We use well known formulas
\begin{equation}
f(z+k)=\sum_{j=0}^k \begin{pmatrix} k  \\ j \end{pmatrix}\Delta^jf(z),
\label{7.2}
\end{equation}
and
\begin{equation}
\Delta^k f(z)=\sum_{j=0}^k \begin{pmatrix} k\\ j \end{pmatrix}(-1)^{k-j}f(z+j).\label{7.3}
\end{equation}
%%%%%%%%%%%%%%
Note that \eqref{summer.eq} is obtained by \eqref{7.2} and \eqref{7.3}. Equation \eqref{7.1} is derived from the following difference equation
\begin{multline}
256 z (z - 1) (z - 2)\Delta^4 y(z-3)+384 z (z - 1)\Delta^3 y(z-2)\\
-80 z\Delta^2 y(z-1)+40\Delta y(z)-81 z (z - 1)y(z-2)=0.\label{7.4}
\end{multline}
In fact, setting $z+3$ in place of $z$ in \eqref{7.4}, we have
\begin{multline}
(256 z^3+ 1920 z^2+4656 z+3640)y(z+4)-(1024 z^3+7296 z^2+16864 z+12616) y(z+3)\\
+(1536 z^3+ 10368 z^2+ 22576 z+15888)y(z+2)-(1024z^3 +6609z^2 +13589z+8934)y(z+1)\\
+(256 z^3+ 1536 z^2+ 2816 z+1536)y(z)=0.\label{7.5}
\end{multline}
Applying \eqref{7.2} to \eqref{7.5}, we obtain \eqref{7.1}.
Consider a solution $y(z)=\sum_{n=0}^\infty a_n z^{\underline{n}}$. % of \eqref{7.4} given by \eqref{4.3} with $\rho=0$.
We note on the first term in \eqref{7.4} that
\begin{align*}
z &(z - 1) (z - 2)\Delta^4 y(z-3)=z (z - 1) (z - 2)\Delta^4 \left(\sum_{n=0}^\infty a_n(z-3)^{\underline{n}}\right)\\
&=z (z - 1) (z - 2)\sum_{n=0}^\infty n(n-1)(n-2)(n-3)a_n (z-3)^{\underline{n-4}}\\
&=\sum_{n=0}^\infty n(n-1)(n-2)(n-3)a_nz^{\underline{n-1}}=\sum_{n=3}^\infty (n+1)n(n-1)(n-2)a_{n+1}z^{\underline{n}}
\end{align*}
Similarly, we compute other terms in \eqref{7.4} and obtain
\begin{align*}
256&\sum_{n=3}^\infty (n+1)n(n-1)(n-2)a_{n+1}z^{\underline{n}}+384\sum_{n=2}^\infty (n+1)n(n-1)a_{n+1}z^{\underline{n}}\nonumber\\
&-80\sum_{n=1}^\infty (n+1)na_{n+1}z^{\underline{n}}+40\sum_{n=0}^\infty (n+1)a_{n+1}z^{\underline{n}}-81\sum_{n=2}^\infty a_{n-2}z^{\underline{n}}=0,
\end{align*}
which gives that $a_1=a_2=0$ and for $n\geq3$
\begin{align}
8(n+1)(2n-1)(4n-5)(4n+1)a_{n+1}-81a_{n-2}=0.\label{7.6}
\end{align}
Hence we see that $a_{3k-2}=a_{3k-1}=0$ for $k\in \mathbb{N}$ and
\begin{equation}
a_{3(k+1)}=\frac{a_{3k}}{(4k+1)(4k+2)(4k+3)(4k+4)}.\label{7.7}
 \end{equation}
This implies that $y(z)$ can be written
\begin{equation*}
 y(z)=\sum_{k=0}^\infty \frac{a_0}{(4k)!}z^{\underline{3k}}.
 \end{equation*}
By means of Theorem~\ref{Convergence}, the formal solution converges to an entire function of order $3/4$, since $\chi(\{a_{3k}\})=\limsup_{k\to\infty}3k\log(3k)/(-\log((4k)!)=3/4<1$.

Observe the corresponding Newton polygon to \eqref{7.1} with $m=4$, $d_4=3$, $d_3=2$, $d_2=1$, $d_1=2$ and $d_0=2$, which gives $p=2$, $s_1=4$ and $s_2=0$. We thus confirmed $\rho=3/4$ by \eqref{orderlist.eq}.
We compute the corresponding recurrent relation to \eqref{recu4.eq} as
    $$
     a_{n+4}Q(n,-4)+a_{n+3}Q(n,-3)+\cdots+a_{n-2}Q(n,2)=0,
    $$
    where
    \begin{align*}
    &Q(n,-4)=8 (n+1) (n+2) (n+3) (n+4) (2 n+5) (4 n+7) (4 n+13)\\
    &Q(n,-3)=24 (n+1) (n+2) (n+3) (2 n+3) (4 n+3) (4 n+9)\\
    &Q(n,-2)=24 (n+1) (n+2) (2 n+1) (4 n-1) (4 n+5)\\
    &Q(n,-1)=(n+1) (256 n^3-465 n^2-357 n-446)\\
    &Q(n,0)=-243 (n+1)(n+2),\quad Q(n,1)=-243 (n+1),\quad Q(n,2)=-81,
        \end{align*}
which gives the corresponding recurrent relation to \eqref{xrelation.eq}.
We obtain the points $(i,j_i)$, $i=0, 1, \dots, 6$ from the first nonzero coefficients described in Section~\ref{ams.sec} as $(0, 0)$, $(1,1)$, $(2,2)$, $(3,3)$, $(4,5)$, $(5,6)$, $(6,7)$. The slope of the segment from $(3,3)$ to $(6,7)$ of the convex broken line implies the order $3/4$.
\end{example}

\begin{example}\label{two.example}

We give a difference equation possessing transcendental entire solutions $d(z)$ and $h(z)$ of order less than 1 whose orders are different, which is a modification of
 \cite[Example~6.1]{IY2004}. First we adopt the \eqref{one.ex} in Example~\ref{one.example}, which gives a transcendental entire solution $f(z)$ of order $\rho(f)=1/3$.
In fact, writing $f(z) = \sum_{n=0}^{\infty} \alpha_n z^{\U{n}}$, we have
\begin{equation}
n(2n-3)(3n-4) \alpha_n = \alpha_{n-1}.\label{6000}
\end{equation}
This implies that $\chi(\{\alpha_n\})=1/3$ and Theorem~\ref{Convergence} concludes that $\rho(f) = 1/3$, see Example~\ref{one.example}, the case $f_1$.

Set $g(z)=zf(z)$. Then $\rho(g)=1/3$, and we see by direct computations that $g(z)$ satisfies an equation of the form
\begin{align}
 L_3[g(z)]&=(6 z^5+37 z^4+84 z^3+83 z^2+30 z)\Delta^3g(z)\nonumber\\
 &\quad -(17 z^4+68 z^3+87 z^2+36 z)\Delta^2g(z)+(33 z^3+97 z^2+66 z)\Delta g(z)\nonumber\\
 &\quad -(z^3+39 z^2+108 z+72)g(z)=0.\label{6001}
\end{align}
Next we consider a difference equation of the fifth order
$\sum_{j=0}^5 c_j^*(z) \Delta^j h(z-j)=0$, where $c_j^*(z)=c_j \cdot z(z-1)\cdots(z-j+1)$, $j=1, 2, \dots, 5$ with $c_5=36$, $c_4=228$, $c_3=271$, $c_2=28$, $c_1=3$, and $c_0^*(z)=c_0z$, $c_0 =-1$.
%Note that $c_1 + c_0 = 2.$
The equation above can be actually written as
\begin{align}
 L_5[h(z)]&=(36 z^4+588 z^3+3583 z^2+9653 z+9702)\Delta^5h(z)\nonumber\\
 & +(228 z^3+2594 z^2+9806 z+12319)\Delta^4 h(z)\nonumber\\
 &+(271 z^2+1981 z+3596)\Delta^3 h(z)\nonumber\\
 & +(28 z+114)\Delta^2 h(z)-2\Delta h(z)-h(z)=0.\label{6002}
\end{align}
Observe the corresponding Newton polygon to \eqref{6002} with $d_5=4$, $d_4=3$, $d_3=2$, $d_2=1$, $d_1=d_0=0$.
We have $s_1=5$ and $s_2=0$, which implies that possible order is $1/5$ by Theorem \ref{two.theorem}.
Further, we write $h(z) = \sum_{n=0}^{\infty} \gamma_n z^{\U{n}}$. By the definitions of $c^*_j(z)$, $j=1, 2, \dots, 5$ and $zh(z) = \sum_{n=1}^{\infty} (n\gamma_n+\gamma_{n-1})z^{\U{n}}$, we have
\begin{equation}
n(2n-1)(2n-3)(3n-1)(3n-4) \gamma_n = \gamma_{n-1}, \quad n \geq 1\label{6003}
\end{equation}
from which we see that $\chi(\{\gamma_n\})=1/5$ and Theorem~\ref{Convergence} concludes that $\rho(h) = 1/5$. Now we define $L_8[y(z)]=L_3[L_5[y(z)]]$. By computations, we derive
\begin{align}\label{6004}
 L_8&[y(z)]\nonumber\\
\quad  =&(216 z^9+7452 z^8+105678 z^7+794461 z^6+3416591 z^5+8524337 z^4\nonumber\\
&\qquad +12085315 z^3+8972550 z^2+2691000 z)\Delta^8y(z)\nonumber\\
 & + (3348 z^8+86148 z^7+870903 z^6+4406121 z^5+11934400 z^4+17615961 z^3\nonumber\\
 &\qquad+13383629 z^2+4084050 z)\Delta^7 y(z)\nonumber\\
  &+ (14130 z^7+248295 z^6+1591736 z^5+4758666 z^4+7634180 z^3\nonumber\\
  &\qquad+6616921 z^2+2403240 z)\Delta^6 y(z)\nonumber\\
    &+ (-36 z^7+14809 z^6+143264 z^5+466401 z^4+1080909 z^3+1284431 z^2\nonumber\\
    &\qquad-112998 z-698544)\Delta^5 y(z)\nonumber\\
      &-(228 z^6+3675 z^5+30092 z^4-1004 z^3+285305 z^2+1168092 z\nonumber\\
      &\qquad +886968)\Delta^4 y(z)\nonumber\\
        &-(277 z^5+3909 z^4+16443 z^3+155801 z^2+394482 z+258912)\Delta^3 y(z)\nonumber\\
          &-(11 z^4+280 z^3+4861 z^2+12576 z+8208)\Delta^2 y(z)\nonumber\\
   &+(-31 z^3-19 z^2+150 z+144)\Delta y(z)\nonumber\\
   &+ (z^3+39 z^2+108 z+72)y(z)=0.
\end{align}
Observe the corresponding Newton polygon to \eqref{6004} with $d_8=9$, $d_7=8$, $d_6=7$, $d_5=7$,
$d_4=6$, $d_3=5$, $d_2=4$, $d_1=3$ and $d_0=3$, which gives $p=3$, $s_1=8$, $s_2=5$ and $s_3=0$.
Obviously, the possible orders which are less than 1 are $\rho_1=1/3$ and $\rho_2=1/5$.
Clearly $h(z)$ above is a transcendental  entire solution of $L_8[y(z)] = 0$ of order $1/5$.
Another entire solution is given by $L_5[d(z)] = g(z)$ where $g(z) = z f(z)$ is the solution of $L_3[g(z)] = 0$ which was obtained above. Write $d(z)=\sum_{n=0}^{\infty} \delta_n z^{\U{n}}$.
By $zf(z) = \sum_{n=1}^{\infty} (n\alpha_n+\alpha_{n-1})z^{\U{n}}$ and \eqref{6003}, we have
\begin{equation}
n(2n-1)(2n-3)(3n-1)(3n-4) \delta_n - \delta_{n-1} = \alpha_{n-1} + n \alpha_n.\label{6005}
\end{equation}
We write $\delta_n = \delta_n^{(1)} + \delta_n^{(2)}$, and obtain $\{\delta_n\}$ satisfying
\begin{align}
  n(2n-1)(2n-3)(3n-1)(3n-4) \delta_n^{(1)} - \delta_{n-1}^{(1)} &= \alpha_{n-1},\label{6006} \\
  n(2n-1)(2n-3)(3n-1)(3n-4) \delta_n^{(2)} - \delta_{n-1}^{(2)} &= n \alpha_n, \label{6007}
  \end{align}
which gives a candidate of entire solution of $L_5[d(z)] = g(z)$.
Let us denote $D_n=\delta_n^{(1)}/\alpha_n$. From \eqref{6000} and \eqref{6006},
\begin{equation*}
(2n-1)(3n-1)D_n=1+D_{n-1}.
\end{equation*}
Further, we define $(2n-1)(3n-1) D_n = D_n'$. Then we have
\begin{multline}
 D_n' = 1 + \frac{1}{(2n-3)(3n-4)} D_{n-1}' \\
  = 1 + \frac{1}{(2n-3)(3n-4)} + \frac{1}{(2n-3)(3n-4)(2n-5)(3n-7)} D_{n-2}' = \cdots ,  \notag
\end{multline}
which implies that $D'_n=O(1)$ and $1/D'_n=O(1)$ as $n\to\infty$.
This yields that $\delta_n^{(1)}= O(n^{-2})\alpha_n$ and $\alpha_n=O(n^2)\delta_n^{(1)}$.
Similarly, we estimate $\delta_n^{(2)}$.  From \eqref{6000} and \eqref{6007},
\begin{equation*}
(2n-1)(3n-1)E_n=\frac{1}{(2n-3)(3n-4)}+E_{n-1}.
\end{equation*}
Further, we define $(2n-1)(3n-1) E_n = E_n'$ and put $E_n=\delta_n^{(2)}/\alpha_n$. Then we have
\begin{multline*}
E_n' = \frac{1}{(2n-3)(3n-4)}+ E_{n-1}' \\
  = \frac{1}{(2n-3)(3n-4)} + \frac{1}{(2n-3)(3n-4)(2n-5)(3n-7)} E_{n-2}' = \cdots ,
\end{multline*}
which implies that $E'_n=O(1)$ as $n\to\infty$, and hence $\delta_n^{(2)}= O(n^{-2})\alpha_n$.
Since they are non-negative, we see that $\chi(\{\delta_n\})$ is of the same order to $g(z)$, which is of the same order to $f(z)$. Hence we have constructed $d(z)$ of a solution of $L_5[d(z)] = g(z)$ that has the same order to $f(z)$.
\end{example}

\bigskip

\medskip
\noindent
\emph{Katsuya Ishizaki}\\
\textsc{The Open University of Japan, 2-11 Wakaba,\\
Mihama-ku, Chiba, 261- 8586 Japan}\\
\texttt{email:ishizaki@ouj.ac.jp}

\medskip
\noindent
\emph{Zhi-Tao~Wen}\\
\textsc{Shantou University, Department of Mathematics,\\
Daxue Road No.~243, Shantou 515063, China}\\
\texttt{e-mail:zhtwen@stu.edu.cn}
\end{document}